\documentclass[11pt,reqno]{amsart} 

\usepackage{amscd}
\usepackage{amsfonts,amssymb,latexsym} 
\usepackage[all]{xy} 

\xyoption{arc}
\CompileMatrices

\newcommand{\udots}{\mathinner{\mskip1mu\raise1pt\vbox{\kern7pt\hbox{.}}
\mskip2mu\raise4pt\hbox{.}\mskip2mu\raise7pt\hbox{.}\mskip1mu}}

\newcommand{\SC}{{\mathcal{C}}}

\newcommand{\SO}{{\mathcal{O}}}

\newcommand{\CC}{\mathbb{C}}

\newcommand{\VV}{\mathbb{V}}

\newcommand{\Ext}{\operatorname{Ext}}

\newcommand{\Hom}{\operatorname{Hom}}

\newcommand{\Pic}{\operatorname{Pic}}

\newcommand{\Sym}{\operatorname{Sym}}

\newcommand{\surj}{\twoheadrightarrow}
\newcommand{\inj}{\hookrightarrow}
\newcommand{\too}{\longrightarrow}

\newtheorem{proposition}{Proposition}[section]
\newtheorem{theorem}[proposition]{Theorem}
\newtheorem{lemma}[proposition]{Lemma}

\theoremstyle{definition}
\newtheorem{definition}[proposition]{Definition}

\numberwithin{equation}{section}

\begin{document}

\title[Poisson structure on moduli of sheaves on a surface]{Poisson structure on the moduli
spaces of sheaves of pure dimension one on a surface}

\author[I. Biswas]{Indranil Biswas}
\address{School of Mathematics, Tata Institute of Fundamental
Research, Homi Bhabha Road, Mumbai 400005, India}
\email{indranil@math.tifr.res.in}

\author[T. L. G\'omez]{Tom\'as L. G\'omez}
\address{Instituto de Ciencias Matem\'aticas (CSIC-UAM-UC3M-UCM),
Nicol\'as Cabrera 15, Campus Cantoblanco UAM, 28049 Madrid, Spain}
\email{tomas.gomez@icmat.es}

\subjclass[2010]{14J60, 53D17, 32J15}

\keywords{Poisson surface, torsion sheaf, Poisson moduli space, symplectic leaf.}

\date{}

\begin{abstract}
Let $S$ be a smooth complex projective surface equipped with a Poisson structure $s$ and
also a polarization $H$.
The moduli space $M_H(S,P)$ of stable sheaves on $S$ having a fixed Hilbert polynomial $P$ of degree one has a natural
Poisson structure given by $s$ \cite{Ty}, \cite{Bo2}. We prove that the symplectic leaves of $M_H(S,P)$ are the fibers
of the natural map from it to the symmetric power of the effective divisor on $S$ given by
the singular locus of $s$.
\end{abstract}

\maketitle

\section{Introduction}\label{sec1}

Let $S$ be a smooth complex projective surface with a polarization $H$. We consider pairs of
the form $(C,\, L)$,
where $C$ is a closed curve on $S$ representing a fixed class $\SC\,=\,[C]$ in the
N\'eron-Severi group ${\rm NS}(S)$ (divisors on $S$ modulo algebraic equivalence),
and
$L$ is a line bundle on $C$ of fixed degree $d$. More generally, fixing a Hilbert polynomial
$$P(t)\,=\,p_1t+p_0\,=\,C\cdot H t +(d+1-g(C))$$ of
degree one, consider the corresponding moduli space $M_H(S,P)$ parametrizing all $H$-stable sheaves on $S$
with Hilbert polynomial $P$. The pairs of the above type are points of such a moduli space $M_H(S,P)$.

Let $-K_S$ denote the anti-canonical line bundle of $S$. Assume that $S$ is equipped with a nonzero
holomorphic section $s$ of $-K_S$. This section $s$ produces a homomorphism
$$
T^*S\, \longrightarrow\, TS
$$
that sends any $w\, \in\, T^*_xS$, $x\, \in\, S$, to the element of $T_xS$ obtained by contracting
$s(x)$ by $w$. This homomorphism is an isomorphism outside the effective divisor on $S$ given by
the section $s$. Moreover, the above
homomorphism $T^*S\, \longrightarrow\, TS$ defines a holomorphic Poisson structure
on $S$ whose singular locus coincides with the divisor given by $s$.

A special class of above pairs $(S,\, s)$ consists of those where $S$ is an
abelian or K3 surface and $s$ is given by a
trivialization of $K_S\,=\, {\mathcal O}_S$. For these pairs the section $s$ is in fact a
symplectic structure on $S$. In this case, Mukai showed that $s$ produces a holomorphic
symplectic structure on $M_H(S,P)$ \cite{Mu}, meaning
a holomorphic closed $2$-form on $M_H(S,P)$ which is fiberwise nondegenerate.

Tyurin in \cite{Ty} generalized the above mentioned result of
Mukai to the general Poisson surface $(S,\, s)$. More precisely,
he proved that $s$ produces a homomorphism
\begin{equation}\label{si}
\sigma\,:\, T^* M_H(S,P) \,\too\, T M_H(S,P)
\end{equation}
which is skew-symmetric.

For $\sigma$ in \eqref{si}, Bottacin proved that the
the Schouten--Nijenhuis bracket $[\sigma,\, \sigma]$ vanishes identically \cite{Bo2}.
In other words, the pairing on locally defined holomorphic functions of $M_H(S,P)$ defined by
$$
(f,\, g)\, \longmapsto\, \{f,\, g\}\, :=\, \sigma(df,\, dg)
$$
satisfies the following conditions:
\begin{itemize}
\item $\{f,\, g\}\,=\, - \{g,\, f\}$,

\item $\{f,\, gh\}\,=\, h\cdot \{f,\, g\} + g\cdot \{f,\, h\}$, and

\item $\{\{f,\, g\}, \, h\}+ \{\{h,\, f\}, \, g\} + \{\{g,\, h\}, \, f\}\,=\, 0$.
\end{itemize}
Therefore, $\sigma$ defines a holomorphic Poisson structure on $M_H(S,P)$.

A symplectic leaf in a Poisson manifold is a maximal sub-scheme
such that the Poisson structure restricts to a symplectic structure.

Let $U\subset M_H(S,P)$ be the open subset corresponding to sheaves
$E$ of the form $E=i_*L$ where $i:C\inj S$ is a smooth curve and $L$
is a line bundle on $C$, and such that the intersection $Z=C\cap D$ is
zero dimensional. We define a morphism $\varphi_0$ from $U$ to the symmetric
product of $D$, sending $E$ to $Z\subset D$. The main result of this article is
to show that the fibers of the morphism $\varphi_0$ are symplectic leaves
(Theorem \ref{thm1}).

\section{A fibration associated to the moduli space}\label{sec2}

We continue with the notation of Section \ref{sec1}. The surface $S$ is equipped with a nonzero
holomorphic section $s$ of $-K_S$. Let $D$ be the effective divisor on $S$ on which the section $s$ vanishes.

Let $E$ be a sheaf corresponding to a point in the moduli space $M_H(S,P)$. This sheaf is of pure dimension
one, because it is stable. The $0$-th Fitting
ideal sheaf $${\rm Fitt}_0(E)\,\subset\, \SO_S$$ of $E$ defines a
sub-scheme $F(E)\,\subset\, S$. It should be mentioned that when $i\,:\,C\,\inj\, S$
is a smooth curve, and $L$ is line bundle on $C$, then
$F(i_*L)\,=\,C$. Note that ${\rm Fitt}_0(E)\,\subset\, {\rm Ann}(E)$
(the annihilator of $E$), but in general this inclusion is not an isomorphism.
Therefore, we have $${\rm Supp}(E)\,\subset\, F(E)\, ,$$ but in general
this inclusion is not an isomorphism. For instance, if we take the
$n$-th direct sum then we have $F((i_*L)^{\oplus n})\,=\,nC$ which is the $n$-th
thickening of $C$, whereas ${\rm Supp}((i_*L)^{\oplus n})\,=\,C$.

Note that we have a disjoint union
$$M_H(S,P)\,=\, \coprod_{\{\SC\in {\rm NS}(S)\, : \, \SC\cdot H=p_1\}} M_H(S,\SC,P)$$
where $M_H(S,\SC,P)$ is the moduli space of stable sheaves $E$ on $S$
with Hilbert polynomial $P$ and $[F(E)]\,=\,\SC\,\in\, {\rm NS}(S)$.

Let $E$ be a sheaf corresponding to a point in the moduli space
$M_H(S,P)$. Then $F(E)\,\subset\, S$ is a curve whose image in the first
Chow group $A^1(S)\,=\,\Pic(S)$ coincides with the first Chern class $c_1(E)$.
Therefore, the moduli space $M_H(S,\SC,P)$ parametrizes all the sheaves $E$
of $M_H(S,P)$ with $[c_1(E)]\,=\,\SC\,\in\, {\rm NS}(S)$.

Fix a class $\SC\,\in\, {\rm NS}(S)$ with $\SC\cdot H\,=\,p_1$.
Set
$$
d'\,=\, \SC\cdot D \, \in\, \mathbb N\, ;
$$
recall that $D$ is the divisor of $s$. Let
$$
M'_H(S,\SC,P) \, \subset\, M_H(S,\SC,P)
$$
be the Zariski open subset where the intersection $F(E)\cap D$
is zero-dimensional. Then there is a map
$$
\varphi\ :\, M'_H(S,\SC,P)\, \longrightarrow\, \Sym^{d'}(D)
$$
that sends any $E\, \in\, M'_H(S,\SC,P)$ to the above intersection
of $F(E)$ with $D$.

If $d'\,=\,0$, we prove that $\sigma$ is an isomorphism and 
$(M'_H(S,\SC,P),\, \sigma)$ is a symplectic variety (Lemma \ref{lem}).

If $d'\,>\,0$, consider the Zariski open subset
$$
M'_H(S,\SC,P)_0\,\subset\, M'_H(S,\SC,P)
$$ 
parametrizing sheaves $E$ with $F(E)$ smooth
(and $F(E)\cap D$ zero-dimensional). In other words,
$E$ is of the form $i_*L$ where
$i\,:\,C\,\inj\, S$ is a smooth curve, and $L$ is a line bundle on the smooth curve $C$
(and $C\cap D$ is zero-dimensional).
Let
$$
\varphi_0\, :\, M'_H(S,\SC,P)_0\, \longrightarrow\, \Sym^{d'}(D)
$$
denote the restriction of $\varphi$ to $M'_H(S,\SC,P)_0$.

We shall prove that the fibers of the above morphism $\varphi_0$ are the symplectic
leaves of the Poisson variety $(M'_H(S,\SC,P)_0,\, \sigma)$.

\section{The set-up}

Let $S$ be a smooth complex projective surface. Fix a polarization $H$ on $S$, meaning the
class of a very ample line bundle on $S$. The canonical
line bundle of $S$ will be denoted by $K_S$. We assume that $-K_S\,=\, K^*_S$ is nontrivial
admitting nonzero sections. Fix a nonzero section
\begin{equation}\label{ds}
0\, \not=\, s\,\in\, H^0(S,\, -K_S)\, ,
\end{equation}
and let $D\, =\, \text{div}(s)$ be the corresponding anticanonical divisor. Note that $D$ is
nonzero because $K_S$ is not trivial.
The section $s$ endows $S$ with a Poisson structure, which is degenerate exactly on the divisor $D$.

Let $P\,=\,P(t)\,=\,p_1t+p_0$ be a polynomial of degree one with integer coefficients,
and let $\SC\,\in\, {\rm NS}(S)$ be a class with $\SC\cdot H\,=\,p_1$.
Let $$M\,=\,M_H(S,\SC,P)$$ be the moduli
space of $H$-stable sheaves $E$ on $S$ such that
\begin{itemize}
\item the Hilbert polynomial of $E$ is $P$, and

\item $[F(E)]\,=\,\SC\,\in\, {\rm NS}(S)$.
\end{itemize}
In particular, these sheaves are pure of dimension one (the degree of the
polynomial $P$). In this article we will consider the Poisson
structure constructed by Tyurin and Bottacin, \cite{Ty}, \cite{Bo2}, on this moduli space $M$.

The above moduli space $M\,=\,M_H(S,\SC,P)$ is a particular case of the moduli space constructed
by Simpson \cite{Si}. We remark that Mukai and Tyurin used the moduli
space of simple sheaves (see \cite{AK} for simple sheaves) but,
since stable sheaves are
simple, the constructions of Mukai and Tyurin also work for the moduli
space of stable sheaves. It should be
mentioned that Mukai and Tyurin also consider sheaves of pure dimension one
\cite[Chapter 2, Section 1]{Ty}.

More concretely, if $i\,:\,C\,\inj\, S$ is a smooth curve in $S$, and $L$ is a 
line bundle of degree $d$ on $C$, we
consider the direct image $$E\,=\, i_* L\, ,$$ which is in fact a torsion sheaf on $S$ supported on
$C$. We are interested in a
moduli space for objects like this, so we fix $P$ to be the Hilbert
polynomial of $E$. This moduli space will parametrize objects $(C,\, L)$ where
both the line bundle $L$ and the support $C$ move. 

Bottacin, \cite{Bo1}, and Markman, \cite{Ma}, have considered Poisson
structure on the moduli space of Higgs pairs on a smooth complex projective curve $X$. These objects
are pairs of the form $(F,\,\theta)$, where $F$ is a vector
bundle on $X$ and $$\theta\, :\, F \,\too \, F\otimes K_X\otimes {\mathcal O}_X(Z)
\,=\, F\otimes K_X (Z)$$ is a homomorphism
with $Z\,\subset\, X$ being an effective divisor. Let $\VV(K_X(Z))$ denote the smooth quasiprojective surface
given by the total space of the line bundle $K_X\otimes {\mathcal O}_X(Z)$ over $X$.
The spectral construction gives an equivalence between semistable Higgs pairs and
semistable sheaves of pure dimension one on $\VV(K_X(Z))$, so their set-up constitutes a special case
of the set-up considered above. It should be mentioned that using the deformation to the normal cone as
done in \cite{DEL}, our moduli space deforms to the moduli space of Higgs pairs.
 
Let $E$ be a sheaf on $S$ corresponding to a point in the moduli
space $M$. Yoneda pairing gives a homomorphism
$$
\Ext^1(E,\,E\otimes K_S)\times \Hom(K_S,\,\SO_S) \,\too\, \Ext^1(E,\,E)\, , 
$$
and therefore, the section $s\,\in\ H^0(-K_S)$ in \eqref{ds} yields a homomorphism
\begin{equation}\label{eqsi}
T^*_EM \,=\, \Ext^1(E,\,E\otimes K_S) \,\stackrel{\sigma}{\too}\, \Ext^1(E,\,E) \,=\,T^{}_EM\, .
\end{equation}

\begin{lemma}\label{lem}
If $\SC \cdot D\,=\,0$, then $\sigma$ is an isomorphism and
$(M'_H(S,\SC,P),\,\sigma)$ is a symplectic variety.
\end{lemma}

\begin{proof}
If $\SC\cdot D\,=\,0$, then $F(E)\bigcap D\,=\,0$ for all $E\,\in\, M'_H(S,\SC,P)$. Let $C\,=\,F(E)$.
It follows that $K_S|_C$ is trivial, and hence $s\,\in\, H^0(S, \Hom(K_S,\, \SO_S))$
induces an isomorphism
$$
E\otimes K_S \,\stackrel{\cong}{\too}\, E
$$
which in turn shows that $\sigma_1$ and $\sigma_2$ in \eqref{localglobal} are
isomorphisms. Therefore, $\sigma$ is also an isomorphism.
\end{proof}

\section{Symplectic subspace}

A Poisson structure on a vector space $T$ is simply a skew-symmetric
homomorphism $\sigma\,:\,T^* \,\too \,T$. Given such a $\sigma$, there is a unique subspace $T_0$ of $T$
such that $\sigma$ factors as
$$
T^* \,\surj\, T_0^*\,\stackrel{\overline\sigma}{\too}\, T_0 \,\inj\, T
$$
and satisfies the condition that $\overline\sigma$ is an isomorphism. In fact, $T_0\,=\, \sigma(T^*)$. The inverse isomorphism 
$${\overline\sigma}^{-1}\,:\,T_0\,\too\, T^*_0$$ defines a symplectic structure on the subspace $T_0$.

\begin{definition}
The {\it symplectic subspace} of a Poisson structure $(T,\,\sigma)$ is 
the subspace $T_0$ endowed with the symplectic structure 
${\overline\sigma}^{-1}$.
\end{definition}

The moduli space $M_H(S,P)$ of stable sheaves defined earlier is smooth. Indeed,
the obstruction space is $\Ext^2(E,\,E)^0$, namely the traceless part of the
Ext group $\Ext^2(E,\,E)$. This $\Ext^2(E,\,E)^0$ is Serre dual to
$\Hom(E,\,E\otimes K_S)$. The nonzero section $s\,\in\, H^0(S,\, -K_S)$ provides an
injection $\Hom(E,\,E\otimes K_S)\,\subset \,\Hom(E,\, E)\,=\,\CC$ (recall
that a stable sheaf $E$ is simple). Therefore, we have $\Ext^2(E,\, E)^0=0$,
proving that the moduli space is smooth. The tangent space $T_E M_H(S,P)$
is $\Ext^1(E,\,E)$.

The following proposition describes the symplectic
subspace associated to the Poisson structure $\sigma$, defined by Tyurin,
on $T_EM\,=\,\Ext^1(E,\,E)$.

\begin{proposition}
If the intersection $F(E)\bigcap D$ is zero-dimensional,
then the homomorphism $\sigma$ in \eqref{eqsi} fits into a commutative diagram
\begin{equation}
\label{bigdiag}
\xymatrix{
0 \ar[r] & {H^1(Hom(E,\,E\otimes K_S))} \ar[r] \ar@{->>}[d]^{\sigma_1}&
{\Ext^1(E,\,E\otimes K_S)} \ar[r]\ar@{->>}[d] &
{H^0(Ext^1(E,\,E\otimes K_S))} \ar[r]\ar@{=}[d] & {0} \\
0 \ar[r] & {H^1(Hom(E,\,E))} \ar[r] \ar@{=}[d]&
{T_0^*} \ar[r]\ar[d]^{\overline\sigma}_{\cong} &
{H^0(Ext^1(E,\,E\otimes K_S))} \ar[r]\ar@{=}[d] & {0} \\
0 \ar[r] & {H^1(Hom(E,\,E))} \ar[r] \ar@{=}[d]&
{T_0} \ar[r]\ar@{^{(}->}[d] &
{H^0(Ext^1(E,\,E\otimes K_S))} \ar[r]\ar@{^{(}->}[d]^{\sigma_2} & {0} \\
0 \ar[r] & {H^1(Hom(E,\,E))} \ar[r]^{i_{E,E}} &
{\Ext^1(E,\,E)} \ar[r] &
{H^0(Ext^1(E,\,E))} \ar[r] & {0}
}
\end{equation}
such that $\sigma$ in \eqref{eqsi} is the composition of the vertical arrows in the middle column
of the above diagram.
\end{proposition}

\begin{proof}
The two following short exact sequences (in \eqref{localglobal}) are given by 
the local-to-global spectral sequence for Ext, and the homomorphism between these
exact sequences is given by the section $s\,\in\, H^0(S, \Hom(K_S,\, \SO_S))$:
\begin{equation}\label{localglobal}
\xymatrix{
0 \ar[r] & {H^1(Hom(E,\,E\otimes K_S))} \ar[r] \ar[d]^{\sigma_1}&
{\Ext^1(E,\,E\otimes K_S)} \ar[r]\ar[d]^{\sigma} &
{H^0(Ext^1(E,\,E\otimes K_S))} \ar[r]\ar[d]^{\sigma_2} & {0} \\
0 \ar[r] & {H^1(Hom(E,\,E))} \ar[r]^{i_{E,E}} &
{\Ext^1(E,\,E)} \ar[r] &
{H^0(Ext^1(E,\,E))} \ar[r] & {0}
}
\end{equation}
We note that the above sequences are exact on the right because the next term is
$H^2(Hom(E,\,E))$, and $H^2(Hom(E,\,E))$ vanishes since the support of $E$ is
one-dimensional. 
The long exact sequence of cohomologies associated to the short exact sequence of sheaves
$$
0 \,\too\, Hom(E,E\otimes K_S) \,\too\, Hom(E,E) \,\too\, Hom(E,E)|_D \,\too\, 0
$$
shows that $\sigma_2$ in \eqref{localglobal} is injective and $\sigma_1$ is surjective (here we use that $F(E)\bigcap D$ is zero-dimensional).
Therefore, the commutative diagram in \eqref{localglobal} has a factorization as
in the statement of the proposition.
\end{proof}

\section{Symplectic leaves}

Let $E$ be a pure sheaf corresponding to a point in the moduli
space $M\,=\,M_H(S,\SC,P)$. Recall from Section \ref{sec2} that
the $0$-th Fitting ideal sheaf
$${\rm Fitt}_0(E)\,\subset\, \SO_S$$
of $E$ defines a sub-scheme $F(E)\,\subset\, S$. The Hilbert polynomial
of $F(E)\,\subset\,S$ depends only on the Hilbert polynomial of $E$.
Let $Q$ denote the Hilbert
scheme of sub-schemes of $S$ with Hilbert polynomial that of $F(E)$.
The above construction defines a morphism
$$
\begin{array}{ccc}
M & {\stackrel{f}\too} & Q\\
E & {\longmapsto} & (F(E)\,\subset\, S).\\
\end{array}
$$
It may be mentioned that this is analogous to the Hitchin map for Higgs bundles. Recall that
a Higgs bundle on a curve is equivalent to a pure sheaf on the total
space of the cotangent bundle of the curve, and the sub-scheme defined
by the $0$-th Fitting ideal of
this pure sheaf is the spectral curve corresponding to the Higgs bundle.

Let $C\,=\,F(E)$ for some $H$-stable sheaf $E$.
We will assume that $C\cdot D\,\neq\, 0$. The obstruction space for smoothness, at the point $C$,
of the Hilbert scheme $Q$ defined above is
$$\Ext^1(\SO_S(-C),\,\SO_C)\,=\,H^1(S,\, \SO_C\otimes \SO_S(C))\, .$$
Note that $H^1(S,\, \SO_C\otimes \SO_S(C))$ is Serre dual to
$H^0(C,\,K_C\otimes i^* \SO_S(-C))\,=\,H^0(C,\, K_S\vert_C)$,
where $i\, :\, C\, \hookrightarrow\, S$ is the inclusion map.
Now $H^0(C,\, K_S|_C)$ vanishes for any effective
curve $C$ with $C\cdot D\,\neq\, 0$, because $K_S\,\cong\, \SO_S(-D)$. Therefore, we conclude
that the point $C$ of the Hilbert scheme $Q$ is smooth.

Let $M'_0\,\subset\, M$ be the open subset corresponding to sheaves $E$
such that $F(E)$ is a smooth curve $C$ whose intersection with
$D$ is zero-dimensional. 
In other words, $M'_0$ is the
pre-image, under $f$, of the open subset $Q_0$ of the Hilbert scheme $Q$
corresponding to smooth curves $C$ whose intersection with
$D$ is zero-dimensional.
Alternatively, $M'_0$ is the open subset corresponding to sheaves of
the form $$E\,=\,i_*L\, ,$$ where $i\,:\,C\,\inj\, S$ is a smooth
curve intersecting $D$ in a zero-dimensional sub-scheme, and $L$ is a
line bundle on $C$.

Consider the composition
$$
\varphi_0\,:\,M'_0 \,\stackrel{f_0}{\too}\, Q_0 \,\stackrel{g}{\too}\, \Sym^{d'}(D)\, ,
$$
where $f_0\,=\,f|_{M'_0}$ is the restriction of $f$ to $M'_0$, and 
$g$ sends a curve $C$ to the sub-scheme $Z\,=\,C\bigcap D$ of $D$;
the sub-scheme $Z$ is nonzero because we are assuming $C\cdot D\,\neq\, 0$.

We will prove that the fibers of $\varphi_0$ are symplectic leaves of the Poisson structure of
$M'_0$, meaning for each point $E$ of $M'_0$, the tangent space $T_E\varphi_0^{}$ to the fiber of
$\varphi_0$ passing through $E$ is the symplectic subspace of $T_EM$.

The following result is analogous to \cite[Corollary 8.10]{Ma} and 
\cite[Theorem 4.7.5]{Bo1}.

\begin{theorem}\label{thm1}
The fibers of the morphism $\varphi_0$ are symplectic leaves of the
Poisson structure of $M'_0$.
\end{theorem}

\begin{proof}
What we need to prove is that the tangent vector subspace 
$T_E\varphi_0^{} \,\subset\, T_E{M'_0}$ to
the fiber of $\varphi_0^{}$ at $E$ coincides with the 
subspace $T_0\,\subset \,\Ext^1(E,\,E)$ in \eqref{bigdiag}, using
the canonical identification $T_E{M'_0}\,=\,\Ext^1(E,\,E)$. The idea of the
proof is to show the following:
\begin{itemize}
\item the bottom row of the diagram in \eqref{bigdiag} is
identified canonically with the short exact sequence for the tangent
spaces given by the morphism $f_0$, and
\item the vertical inclusion on the right
column of the diagram in \eqref{bigdiag} is identified with the inclusion of tangent space to
the fiber of $g$.
\end{itemize}

The fiber of $f_0$ over a point of $Q$ corresponding to a (smooth)
curve 
$C\subset S$ is the Picard variety $\Pic^{d}(C)$ 
of line bundle on $C$ of degree $d$.
The embedding of this fiber on $M'_0$ is given by the morphism
$$
\begin{array}{ccc}
\Pic^{d}(C) & \stackrel{\psi_C}{\too} & M'_0 \\
L & \longmapsto & E=i_* L \\
\end{array}
$$
The differential of this morphism $\psi_C$ 
at a point corresponding to a line
bundle $L\, \in\, \Pic^{d}(C)$ is
$$
d\psi_C\,:\, H^1(C,\,\SO_C)\,=\,T_E{f_0} \,\too\, T_E{M'_0}\,=\,\Ext^1(E,\,E)\, .
$$
Since $f_0$ is a smooth morphism 
(note that $E$ is a smooth point of $M_H(S,P)$ and
the fiber of $f$ containing $E$ is the Picard variety $\Pic^{d}(C)$), 
we have a short exact sequence
relating the tangent bundles of $M'_0$ and $Q_0$
$$
0 \,\too\, T{f_0} \,\stackrel{d f_0}{\too}\, T{M'_0} \,\too\, f^*_0T{Q_0} \,\too\, 0
$$
which gives a short
exact sequence of the fibers
\begin{equation}\label{diff}
0 \,\too\, T_E{f_0} \,\stackrel{d f_0|_E}{\too}\, T_E{M'_0} \,\too\,T_{C}{Q_0} \,\too \,0\, .
\end{equation}
Using the following canonical isomorphisms
$$
H^1(C,\,\SO_C)\,=\, H^1(C,\,End(L,\,L)) \,=\, H^1(S,\,i_*(End(L,\,L))) \,=\, H^1(S,\,End(E,\,E))
$$
the differential $d f_0|_E$ is identified with the first map in the short
exact sequence coming from the local-to-global spectral sequence for
Ext, that is $d f_0|_E\,=\,i_{E,E}$ in \eqref{localglobal}. 
Therefore, the sequence in \eqref{diff} is
identified with the
local-to-global exact sequence (bottom row of diagram \eqref{bigdiag}).

The quotients of both sequences are
identified by the canonical isomorphisms
$$
T_CQ_0\,=\,\Hom(\SO_S(-C),\,\SO_C) \,=\, H^0(\SO_C(C)) \,=\, H^0(N_{C/S})\, ,
$$
where $N_{C/S}$ is the normal bundle of $C$ in $S$, and
$$
Ext^1(E,\,E)\,=\,Ext^1(i_*L,\,i_*L)\,=\,Ext^1(\SO_C,\,i_*(L^*\otimes L))\,=\, 
Ext^1(\SO_C,\,\SO_C) \,=\, N_{C/S}\, ,
$$
so that 
$$
H^0(S,\,Ext^1(E,\,E)) \,=\, H^0(N_{C/S})\, .
$$
The term at the top right of \eqref{bigdiag} is canonically identified as
$$
H^0(S,\, Ext^1(E,\,E\otimes K_S))\,=\,H^0(S,\, Ext^1(E,\,E)\otimes \SO_S(-D))
$$
$$
=\, H^0(S,\, N_{C/S}\otimes \SO_S(-D))\,=\,H^0(N_{C/S}(-D))\, .
$$
This is the subspace of $H^0(N_{C/S})\,=\,T_CQ_0$ corresponding to deformations
of $C$ which fix the intersection $C\cap D \,=:\,Z$, i.e., the tangent
space $T_Cg$ to the fiber of $g$ containing $C$.
Therefore, the commutative diagram
$$
\xymatrix{
0 \ar[r] & 
{T_E{f_0}} \ar[r] \ar@{=}[d]&
{T_E\varphi_0^{}} \ar[r]\ar@{^{(}->}[d] &
{T_Cg} \ar[r]\ar@{^{(}->}[d] & {0} \\
0 \ar[r] & 
{T_E{f_0}} \ar[r]^{d\psi_C} &
{T_E{M'_0}} \ar[r] &
{T_{C}{Q_0}} \ar[r] & {0}
}
$$
is identified with the bottom two rows of the diagram in \eqref{bigdiag}, and then
$T_E\varphi_0^{}\,=\,T_0$.
\end{proof}

\section*{Acknowledgements}

We thank N. Fakhruddin and O. Villamayor for useful discussions.
We thank the International Center for Theoretical Sciences at Bangalore
for hospitality while this work was being completed. The first author
is partially supported by a J. C. Bose Fellowship. The second
author is partially supported by the European Union
(Marie Curie IRSES Fellowship within the 7th Framework Programme
under agreement 612534 MODULI), and
Ministerio de Ciencia, Innovaci\'on y Universidades
(grants MTM2016-79400-P and
ICMAT Severo Ochoa project SEV-2015-0554).

\end{document}